\documentclass [11pt,twoside,a4paper]{article}
\usepackage{amsfonts}
\usepackage{amsthm}
\usepackage{amsmath}
\usepackage{amscd}
\usepackage{amsbsy}            %
\usepackage{psfrag}            %
\usepackage{epsf}              
\usepackage{graphicx}          %
\usepackage{makeidx}           %
\usepackage{color}             %
\usepackage{fancyhdr}

\newtheorem{thm}{Theorem}[section]

\newtheorem{lem}[thm]{Lemma}
\newtheorem{definition}[thm]{Definition}

\newtheorem{prop}[thm]{Proposition}
\newtheorem{cor}[thm]{Corollary}

\newtheorem{conj}[thm]{Conjecture}

\newcommand{\cov}{\texttt{cov}}
\newcommand{\tr}{\texttt{Tr}}

\newcommand{\abs}[1]{\left\vert#1\right\vert}
\newcommand{\norm}[1]{\parallel\! #1\! \parallel}
\newcommand{\normal}{\mathcal{N}}
\newcommand{\seq}[1]{\left<#1\right>}
\newcommand{\set}[1]{\left\{#1\right\}}

\newcommand{\trans}{\texttt{Trans}}
\newcommand{\var}{\texttt{Var}}
\newcommand{\vecc}{\texttt{vec}}

\newcommand{\al}{\alpha}

\newcommand{\be}{\beta}
\newcommand{\ga}{\gamma}

\newcommand{\om}{\omega}
\newcommand{\Om}{\Omega}
\newcommand{\la}{\lambda}

\newcommand{\p}{\prime}
  \newcommand{\ip}{i^{\p}}
  \newcommand{\jp}{j^{\p}}

\newcommand{\si}{\sigma}
\newcommand{\Sig}{\Sigma}
\newcommand{\tta}{\theta}
\newcommand{\ifff}{\Leftrightarrow}

\newcommand{\A}{\mathcal{A}}
\newcommand{\B}{\mathcal{B}}
\newcommand{\C}{\mathcal{C}}
\newcommand{\D}{\mathcal{D}}

\newcommand{\X}{\mathcal{X}}
\newcommand{\Y}{\mathcal{Y}}
\newcommand{\Z}{\mathcal{Z}}
\newcommand{\K}{\mathcal{K}}
\newcommand{\M}{\mathcal{M}}

\newcommand{\N}{\mathcal{N}}
\newcommand{\cH}{\mathcal{H}}

\newcommand{\ST}{\mathcal{S\!T}}
\newcommand{\T}{\mathcal{T}}

\newcommand{\bft}{\textbf{t}}

\newcommand{\bv}{\textbf{v}}

\newcommand{\bx}{\textbf{x}}
\newcommand{\bX}{\textbf{X}}

\newcommand{\by}{\textbf{y}}

\newcommand{\bfd}{\rm{d}}

\newcommand{\bfm}{\textbf{m}}

\def\R{\mathbb R}

\newcommand{\mnrt}{$m$th order $n$-dimensional real tensor }
\newcommand{\mnrts}{$m$th order $n$-dimensional real tensors }

\newcommand{\mnsts}{$m$th order $n$-dimensional symmetric tensors }

\newcommand{\beq}{\begin{equation}}
\newcommand{\eeq}{\end{equation}}
\newcommand{\bey}{\begin{eqnarray}}
\newcommand{\eey}{\end{eqnarray}}
\newcommand{\beyy}{\begin{eqnarray*}}
\newcommand{\eeyy}{\end{eqnarray*}}

\newcommand{\mb}[1]{\text{\mathversion{bold}${#1}$}}

\newcommand{\ba}{\begin{array}}
\newcommand{\ea}{\end{array}}
\newcommand{\bt}{\begin{tabular}}
\newcommand{\et}{\end{tabular}}
\newcommand{\etb}{\end{table}}
\newcommand{\bc}{\begin{center}}
\newcommand{\ec}{\end{center}}
\newcommand{\Bea}{\begin{eqnarray*}}
\newcommand{\Eea}{\end{eqnarray*}}

\title{Random Tensors and their Normal Distributions}
\author{Changqing Xu\thanks{Corresponding author. School of Mathematical Sciences, Suzhou University of Science and Technology, Suzhou, China
Email: cqxurichard@mail.usts.edu.cn}   
\and 
Kaijie Xu\thanks{School of Electronic Engineering, Xidian University, 710071, Xi'an, China. Email: kjxu@xidian.edu.cn}
}

  \makeatletter
      \def\@setcopyright{}
      \def\serieslogo@{}
      \makeatother
 \date{\today}
\begin{document}
\maketitle

\begin{abstract}
The main purpose of this paper is to introduce the random tensor with normal distribution, which promotes the matrix normal distribution to a higher order case.  
Some basic knowledge on tensors are introduced before we focus on the random tensors whose entries follow normal distribution. The random tensor 
with standard normal distribution(SND) is introduced as an extension of random normal matrices. As a random multi-array deduced from an affine 
transformation on a SND tensor, the general normal random tensor is initialised in the paper.  We then investigate some equivalent definitions of a normal tensor 
and present the description of the density function, characteristic function, moments, and some other functions related to a random matrix.   A general form of an even-order 
multi-variance tensor is also introduced to tackle a random tensor.  Finally some equivalent definitions for the tensor normal distribution are described.  
\end{abstract}

\noindent \textbf{keywords:} \  Tensor; mixed effect tensor model; parameter estimation; Normal distribution;  Characteristic function.\\
\noindent \textbf {AMS Subject Classification}: \   53A45, 15A69.  \\


\section{Introduction}
\setcounter{equation}{0}

The multivariate statistics have been used in many areas including the medical imaging. The classical treatment of a normal random matrix $X$ is to use the traditional 
multivariate normal distribution after the vectorization of $X$. Basser has studied the high-order diffusion tensor (DT) of water through the magnetic resonance image (MRI), 
called the DT-MRI, since 1999 \cite{Base1999,Base2003,Base2007}.  As Basser pointed out in \cite{Base2003}, the traditional approach to tackle normal random matrices 
does not satisfy the requirement of the DT-MRI for the distribution and moments of both its components and its eigenvalues and eigenvectors. However Basser and his cooperators stay in the matrix style to describe the properties of those tensors, which somehow impeded further investigation of DT-MRI.  \\
\indent  The systematic treatment of multivariate statistics through matrix theory has been developed since 1970s\cite{CC1970,WS1972}. In multivariate statistics, the $k$-moment of a random vector $\bx\in \R^{n}$ is conventionally described by a matrix for any positive integer $k>1$. The \emph{multivariate normal distribution}, also called the \emph{joint normal distribution}, usually deals with random vectors with normal distributions. The study of random 
matrices, motivated by quantum mechanics and wireless communications \cite{CD2011,F2010,M2004,MS2005} etc. in the past 70 years, mainly focuses on the spectrum 
properties\cite{MS2005} and can be used in many areas such as the classical analysis and number theory using enumerative combinatorics\cite{MS2005}, Fredholm 
determinants\cite{M1992}, diffusion processes\cite{BGL2017}, integrable systems\cite{BB1999}, the Riemann Hilbert problem\cite{JLYY2011} and the free probability theory 
in the 1990s. The theory of random matrices in multivariate statistics basically focuses on the distribution of the eigenvalues of the matrices\cite{Tao2012} .\\ 
\indent  Recently we found some 
\indent The need to use the high order tensors has manifested in many areas more than half a century ago, and the recent growing development of multivariate distribution 
theory poses new challenge for finding some novel tools to describe classical statistical concepts e.g. moment, characteristic function and covariance etc. This in turn has 
facilitated the development of higher order tensor theory for multilinear regression model\cite{cook2010} and the higher order derivatives of distribution functions. Meanwhile, 
the description of an implicit multi-relationship among a family of random variables pose a challenge to modern statisticians. The applications of the high order tensors in 
statistics was initialized by Cook etc. \cite{cook2010,cook2013} when the envelope models were established. \\
\indent  In this paper, we first use tensors to express the high order derivatives, which in turn leads to the simplification of the high-order moments and the covariances of a random matrix. We also introduce the normal distributions of a random matrix as well as that of a random tensor. The Gaussian tensors are investigated to extend the random matrix theory. \\  
\indent  By a random vector $\bx:=(x_{1}, \ldots, x_{n})^{\top}\in \R^{n}$ we mean that each component $x_{i}$ is a random variable (r.v.). Here we usually do not distinguish a 
row and a column vector unless specifically mentioned. There are several equivalent definitions for a random vector to be Gaussian.  Given a constant vector $\mu\in \R^{n}$ 
and a positive semidefinite matrix $\Sig\in \R^{n\times n}$. A random vector $\bx\in \R^{n}$ is called a \emph{Gaussian} or \emph{normal} vector with parameter $(\mu, \Sig)$ if 
it is normally distributed with $E[\bx]=\mu$ and $\var(\bx)=\Sig, \al\in \R^{n}$. This is equivalent to a single variable normal distribution of $\al^{\top} \bx$ for all $\al\in \R^{n}$. It 
is obvious from this fact that each component of a normal vector is normal. The converse is however not true.  A random tensor $\A=(A_{i_{1}i_{2}\ldots i_{m}})$ is an 
$m$-order tensor whose entries are random variables.  As a special case, a random matrix is a matrix whose entries are random variables.  \\
\indent The covariance matrix of a random vector $\bx$ restores the variances and covariances of its coordinates and plays a very important role in the statistical analysis. 
However, it cannot demonstrate the multi-variances of a group of variables.  On the other hand, the $k$-moment of a random vector $\bx\in \R^{n}$ can be defined through 
the high-order derivative of its characteristic function, which can be depicted much more clearly by a high-order tensor.  Note that an $m$-order derivative of a function 
$f(\bx)=f(x_{1},x_{2},\ldots, x_{n})$, usually defined as an $n\times n^{m-1}$ matrix, can be expressed as a symmetric $m$-order $n$-dimensional tensor.  A direct bonus of 
a tensor form of derivative is that we can locate and label easily any entry of $\frac{\partial^{m} f}{\partial \bx^{m}}$.\\ 
\indent  We denote $[n]$ for the set $\set{1,2,\ldots, n}$ for any positive integer $n$ and $\R$ the field of real numbers. Throughout we follow the convention to use the lowercase 
italic letters e.g. $a,b,\cdots,$ for scalars, uppercase italic letters e.g. $A,B,\ldots $  for matrices, lowercase boldface upright  letters (e.g. $\bx, \by, \cdots, $ )  for random vectors while 
the lowercase italic letter such as $x,y,\cdots,$ for their values, sample points or observations, the uppercase script letters e.g. $\A, \B, \X, \Y, \cdots $ to denote tensors. We reserve the 
letters $X, Y, Z$ (in either case or style) strictly for random objects and all other letters for deterministic objects.  Finally a moment of a random vector is denoted by $\bfm$ which 
could be either a scalar, a vector, a matrix or a tensor. \\ 
\indent   Given a matrix $A\in \R^{m\times n}$, we use $\vecc(A)$ to denote the $mn$-dimensional column vector $\bv=(v_1,v_2,\ldots, v_{mn})^{\top}$ formed by stacking  
the columns of $A$, i.e.,  $\bv^{\top} =(\beta_1^{\top}, \beta_2^{\top},\ldots, \beta_n^{\top})$  where $A=[\beta_1, \beta_2, \cdots, \beta_n]$.  A tensor $\A$ of size 
 $\bfd:=d_1\times d_2\times \ldots \times d_m$ is an \emph{$m$-way matrix} or an $m-$order tensor.  $\A$ is called an \mnrt\  tensor if  $d_1=\ldots =d_m=[n]$.  
 Let $\T({\rm d})$ be the set of all the $m$-order tensors indexed by $\bfd$, $\T_{m;n}$ be the set of all \mnrts, and $\T_{m}$ be the set of all $m$-order tensors.  
 Thus a scalar, a vector and a matrix is respectively a tensor of order zero, one and two.  For a tensor $\A$, we usually denote by $A_{i_1i_2\ldots i_p}$ or $A_{\si}$ 
 the component of $\A$  associated with the indices $\si:=(i_1,i_2,\ldots, i_p)$.  A tensor $\A\in \T_{m;n}$ is said to be a \emph{symmetric tensor} if each entry of $\A$ is 
 invariant for any permutation of its indices.  An \mnrt $\A\in \T_{m;n}$ is associated with an $m$-order $n$-variate homogeneous polynomial in the form   
\beq\label{def: asspolyn} 
f_{\A}(\bx) = \A\bx^m = \sum_{i_1,i_2,\ldots,i_m} A_{i_1i_2\ldots i_m} x_{i_1}x_{i_2}\cdots x_{i_m} 
\eeq
where the summation is taken over the index set
\[ S(m,n):= \set{(i_1,i_2,\cdots,i_m): i_k\in [n],  \forall k\in [m] } \]
We denote by $\ST_{m;n}$ the set of all \mnsts.  A symmetric tensor $\A\in \ST_{m;n}$ is said to be \emph{positive semidefinite} (PSD)  if  $f_{\A}(\bx):=\A\bx^m>0$
($\ge 0$) for all $0\neq \bx\in \R^n$, and $\A$ is called a \emph{copositive} tensor if  $f_{\A}(\bx)\ge 0$ for all nonnegative vector $\bx$. For the study of the symmetric 
tensors, PSD tensors, the copositive tensors, including their spectrum, decompositions and other properties, we refer the reader to \cite{CS2013, cglm08, qieig2005,qicop2013} and \cite{cglm08}. \\
\indent  Let $I:=I_{1}\times \ldots \times I_{m}$ where each $I_{k}$ represents an index set (usually a set in form $\set{1,2,\ldots, n_{k}}$) and $\A\in  \T({\rm I})$.  
For any $k\in [m]$ and $j\in I_{k}$, $\A$'s $j$-slice along the $k$-mode (denoted by $A^{(k)}[j]$) is defined as an $(m-1)-$order tensor 
$(A_{i_{1}\ldots i_{k-1} j  i_{k+1}\ldots i_{m}})$ where the $k$th subscript of each entry of $\A$ is fixed to be $j\in [I_{k}]$.  
An $m$-order tensor can be sliced into a series of $(m-1)$-order tensors along any of the $m$ modes.   A high order tensor can be flattened or unfolded into a matrix 
by slicing iteratively.  For example, an $m\times n\times p$ tensor $\A$ can be unfolded to be a matrix $A[1]\in \R^{m\times np}$ along the first mode and $A[2]\in \R^{n\times pm}$ if along the second mode. There are ten options to flatten a 4-order tensor $m\times n\times p\times q$  into a matrix: four to reserve one mode and stack the other three and six to group two modes together to form a matrix.\\
\indent  The product of tensors can be defined in many different ways. Given any tensors $\A,\B$ of appropriate size, the $k$-mode contractive product of $\A$ and $\B$
w.r.t. the chosen mode(s). This can be regarded as a generalisation of the matrix product. For more detail, we refer the reader to \cite{qiluo2017} and \cite{CC2010}. \\
\indent  Given a random vector $\bx\in \R^{n}$.  The characteristic function (CF) of $\bx$, is defined by 
\[ \phi_{\bx}(\bft)= E[exp(\imath \bft^\p \bx)], \forall \bft\in \R^n\]
For any positive integer $k$ , the $k$-\emph{moment} of a random vector $\bx$ is defined by 
\beq\label{eq:defkm4vec} 
\bfm_k(\bx) = \frac{1}{\imath^k} \frac{d^k}{d\bft^k} \phi_{\bx} (\bft) |_{\bft=0} 
\eeq
Note that $\bfm_{1}(\bx)=\frac{d}{d\bft} \phi_{\bx} (\bft)\in \R^{n}$ is a vector and $\bfm_{k}(\bx) =\frac{d^k}{d\bft^k} \phi_{\bx} (\bft)\in \R^{n\times n^{k-1}}$ is a matrix 
for each $k>1$ by the conventional definition. \\
\indent To simplify the definition of the $k$-moments of a random vector, we present a tensor form of the high order multivariate function. Let $\by=f(\bx)$ be a mapping 
from $\C^{n}$ to $\C^{m}$, i.e., $\by=(y_{1},\ldots, y_{m})^{\top}\in \C^{m}$ with each components $y_{i}=f_{i}(\bx)$ sufficiently differentiable. Then we denote 
$H(\by,\bx)=(h_{ij})$ as the Jacobi matrix of $\by$ w.r.t. $\bx$ defined by $h_{ij} :=\frac{dy_{i}}{dx_{j}}$ for all $i\in [m], j\in [n]$. Thus  $H(\by,\bx)\in \R^{m\times n}$. 
Now we define the $k$-order differentiation tensor by 
\beq\label{eq:defkdiff}
\cH^{k}(\by, \bx)=(h_{ij_{1}j_{2}\ldots j_{k}})
\eeq  
which is an $(k+1)$-order tensor of size $m\times \overbrace{n\times \ldots \times n}^{k}$ where 
\beq\label{eq:defkdiff2}  
h_{ij_{1}j_{2}\ldots j_{k}} = \frac{\partial^{k} y_{i}}{\partial x_{j_{1}} \partial x_{j_{2}} \ldots \partial x_{j_{k}}}  
\eeq
for any $i\in [m], j_{1},\ldots, j_{k}\in [n]$. Recall that the conventional form for the $k$-order differentiation of a mapping $\by=f(\bx)$ produces an $m\times n^{k}$ matrix 
$H^{k}(\by, \bx)=(h_{ij})$ where $h_{ij}=\frac{\partial^{k} y_{i}}{\partial x_{j_{1}} \partial x_{j_{2}} \ldots \partial x_{j_{k}}}$ with 
$j= \sum\limits_{s=1}^{k} j_{s} n^{k-s}$ ($j=1,2,\ldots, n^{k}$), making the location of each entry ambiguous.\\ 
\indent  In this paper, we use tensor form to simplify high order differentiations. The $k$-moment of a random vector $\bx\in \R^{n}$ is defined by  
\beq\label{eq: ktensorprodx} 
\bfm_{k}[\bx]=E[\bx^{k}]=E[\overbrace{\bx\times \bx\times \ldots \times \bx}^{k}] 
\eeq
This is a natural extension of the $k$-moment in the univariate case since  
\[ \bfm_{1}[\bx]=E[\bx] \in \R^{n}, m_{2}[\bx]=E[\bx \bx^{\top}] \in \R^{n\times n}, \cdots, \bfm_{k}[\bx] =E[\bx^{k}] \in \T_{m;n}. \]
The definition is identical to the one through characteristic function in the tensor form, as in the following. 
\begin{lem}\label{le: equivdef4moment}
Let $k$ be any positive integer and $\bx\in \R^{n}$ be a random vector with characteristic function $\phi_{\bx}(\bft)$ (with $\bft\in \R^{n}$). Then  
\beq\label{eq:equivdef4m} 
E[\bx^k] = \frac{1}{\imath^k} \frac{d^k}{d \bft^k} \phi_{\bx} (\bft) |_{\bft=0}  
\eeq
where $0$ is a zero vector in $\R^{n}$.    
\end{lem}

\begin{proof}
It is easy to see from the definition of the characteristic function $\phi_{\bx}(\bft)$ and (\ref{eq:defkdiff}) and (\ref{eq:defkdiff2}) that 
\beq
\frac{d^k}{d \bft^k} \phi_{\bx} (\bft) = \imath^{k} E[\exp\{\imath \bft^{\top}\bx\} \bx^{k}] 
\eeq
Thus (\ref{eq:equivdef4m}) holds. 
\end{proof}

\indent  Similarly we can also simplify the definition of the $k$-\emph{central moment} $\bar{\bfm}_k[\bx]$ by $\bar{\bfm}_k[\bx] = E[(\bx -E[\bx])^k]$. Note that our definition is 
consistent with the traditional one for $k\le 2$.  In the following section, we will extend the $k$-moment, the characteristic function, and the related terminology to the case 
for the random matrices. \\

\section{The tensor forms of derivatives of matrices}
\setcounter{equation}{0}

\indent  Let $\tta:= \set{ \tta_{k}: k=1,2,\ldots, p}$ be a 2-partition of $[2p]$ ($\abs{\tta_{k}}=2$ for each $k$), and let $A^{(k)}\in \R^{m_{k}\times n_{k}}, k\in [p]$. 
We may assume w.l.g. that $\tta_{k}:=\set{a_{k},b_{k}}$ with $a_{k}<b_{k}$ for each $k$. The outer product of the matrix sequence $\set{A^{(k)}}$ along partition 
$\tta$ is the $2p$-order tensor $\A$, denoted $\A=A^{(1)}\times_{\tta_{1}}A^{(2)}\times_{\tta_{2}}\ldots \times_{\tta_{p-1}}A^{(p)}$, with 
\[ 
A_{i_{1}j_{1}i_{2}j_{2}\ldots i_{p}j_{p}} = A^{(1)}_{i_{a_{1}}i_{b_{1}}} A{(2)}_{i_{a_{2}}i_{b_{2}}}\ldots A^{(p)}_{i_{a_{p}}i_{b_{p}}}
\]
For $p=2$, we let $\tta_{1}:=\set{s,t}\subset [4]$ with $s<t$ and $\tta_{2}:=\set{p,q}$ the complement of $\tta_{1}$ with $p<q$. The outer product $A\times_{\tta_{1}} B$ 
is the 4-order tensor with the $(s,t)$-modes attributed to $A$ and $(p,q)$-modes attributed to $B$. Thus 
\[ 
(A\times_{(1,2)} B)_{i_{1}i_{2}i_{3}i_{4}} =A_{i_{1}i_{2}}B_{i_{3}i_{4}},\   (A\times_{(1,3)} B)_{i_{1}i_{2}i_{3}i_{4}} =A_{i_{1}i_{3}}B_{i_{2}i_{4}}.
\] 
For  $\set{(s,t)}=\set{(1, 2)}$, we simply write $A\times B$ instead of $A\times_{(1,2)} B$. For a tensor $\A\in \R^{m\times n\times p\times q}$ and a matrix 
$B\in \R^{p\times q}$, the product $C=\A B$ refers to the contractive product, which is defined as 
\[ C_{ij} =\sum\limits_{i^{\p}, j^{\p}} A_{iji^{\p}j^{\p}}B_{i^{\p}j^{\p}}. \]
Note that the product $\A\times_{4} B$ is still a 4-order tensor since 
\[ (\A\times_{4} B)_{i_{1}i_{2}i_{3}i_{4}} = \sum\limits_{k} A_{i_{1}i_{2}i_{3}k} B_{k i_{4}}. \]
We denote $A^{[2]}:=A\times A$ and $A^{(2)}=A\times_{(2,4)} A$ for a matrix $A\in \R^{m\times n}$, i.e.,    
\[ A^{[2]}_{i_{1}i_{2}i_{3}i_{4}} = A_{i_{1}i_{2}}A_{i_{3}i_{4}}, \quad  A^{(2)}_{i_{1}i_{2}i_{3}i_{4}} = A_{i_{1}i_{3}}A_{i_{2}i_{4}}. \]
It is obvious that $A^{[2]}$ has size $m\times n\times m\times n$ while $A^{(2)}$ has size $m\times m\times n\times n$.  \\

\begin{prop}\label{prop:assoc4tensorprod}
Let $A,B,C,D$ be any matrices of appropriate sizes and $\set{s,t}\subset [4]$ and $\set{p,q}=\set{s,t}^{c}$.  Then 
\begin{description}
  \item[(1) ]  $(A\times_{(s,t)} B)\times_{(s,t)} C= (B,C)A$ where $B,C$ are of same size and $(X, Y)$ stands for the inner product of $X,Y$.
  \item[(2) ]  $(A\times_{(s,t)} B)\times_{(p,q)} C=  A\times_{(s,t)} (B\times_{(p,q)} C)$.
  \item[(3) ]  $(A\times B)\times_{4} C = A\times (BC)$.   
  \item[(4) ]  $(A\times_{(s,t)} B)(C\times_{(s,t)} D) = (AC) \times_{(s,t)} (B\times_{(s,t)} D)$.
\end{description}   
\end{prop}
\begin{proof}
(1). We may assume that $A\in \R^{m\times n}, B,C\in \R^{p\times q}$ (the equation is valid only if $B,C$ are of the same size). For simplicity, we let 
$\set{s,t}=\set{1,2}$ and denote $D=(A\times_{(1,2)} B)\times_{(1,2)} C$. Then for any pair $(i,j)$, we have by definition 
\beyy\label{eq: prfprop01}
D_{ij} &=& \sum\limits_{\ip,\jp} (A\times_{(1, 2)} B)_{ij\ip \jp} C_{\ip\jp} \\
         &=& \sum\limits_{\ip,\jp} A_{ij}B_{\ip \jp} C_{\ip\jp} =(B,C)A_{ij}   
\eeyy
which immediately implies (1) for $(s,t)=(1,2)$. Similarly we can also show its validity for other cases.  The second and the third item can also be checked using the same 
technique. To show the last item, we consider the case $(s,t)=(1,3)$ and rewrite it in form  
\beq\label{eq: prop104}
(A_{1}\times A_{2})(B_{1}\times B_{2}) = (A_{1}B_{1})\times (A_{2}B_{2})
\eeq    
where $\times :=\times_{(1,3)}, A_{i}\in \R^{m_{i}\times n_{i}}, B_{i}\in \R^{n_{i}\times p_{i}}$ for $i=1,2$. Denote the tensor of the left hand side and the right hand side 
resp. by $L$ and $R$. Then we have 
\beyy\label{eq: prfprop04}
R_{i_{1}i_{2}i_{3}i_{4}} &=& (A_{1}B_{1})_{i_{1}i_{3}} (A_{2}B_{2})_{i_{2}i_{4}}\\
                                     &=& (\sum\limits_{j=1}^{n_{1}} a_{i_{1}j}^{(1)} b_{j i_{3}}^{(1)})(\sum\limits_{k=1}^{n_{2}} a_{i_{2}k}^{(2)} b_{k i_{4}}^{(2)})\\
                                     &=& \sum\limits_{j,k} a_{i_{1}j}^{(1)} a_{i_{2}k}^{(2)} b_{j i_{3}}^{(1)}b_{k i_{4}}^{(2)}\\
                                     &=& \sum\limits_{j,k} (A_{1}\times A_{2})_{i_{1}i_{2}jk} (B_{1}\times B_{2})_{jki_{3}i_{4}}= L_{i_{1}i_{2}i_{3}i_{4}} 
\eeyy
for all possible $i_{1},i_{2},i_{3},i_{4}$. Thus (\ref{eq: prop104}) holds. This argument can also be extended to other cases.  
 \end{proof} 
 
\indent The outer product of matrices can be extended to the case for any number of  of any order.  Let $\A_1, \A_2, \cdots, \A_s$ be tensors of orders $p_1, p_2, \ldots, p_s$ 
respectively and let $S_{1}, S_2, \cdots, S_s$ be a partition of set $[p], p=p_1+p_2+\cdots + p_k$ where  $\abs{S_j}= p_j, \forall  j\in [s]$.  Denote this partition by 
$\gamma:=\set{S_{1}, S_2, \cdots, S_s}$ for our convenience. The outer product of  $\A_1, \A_2, \cdots, \A_s$ under partition $\gamma$ as 
\beq\label{eq: outprodmultiple} 
(\A_1\times \A_2\times \cdots \times \A_s)_{i_1i_2\ldots i_p} = (\A_1)_{i_{t_{11}}\cdots i_{t_{1p_1}}} (\A_2)_{i_{t_{21}}\cdots i_{t_{2p_2}}}\cdots 
(\A_s)_{i_{t_{s1}} \cdots i_{t_{s p_s}}}    
\eeq
outer productwhere $S_j :=\set{t_{j1}, t_{j2}, \cdots ,t_{j p_j}}$ with entries in increasing order.  Here we allow some $p_j$s to be zero, i.e., $\A_j$ is a scalar, which corresponds to an empty 
underlying set $S _j$.  We denote this outer product of $\A_1, \A_2, \cdots, \A_s$ under (index) partition $\gamma$ by $[\A_1,\A_2,\cdots, \A_s]_{\gamma}$. Note that the 
assumption of the increasing order of the index can be dropped when $\A_j$ is symmetric.  We note that when two tensors $\A_k,\A_l$ agree, the order of the sets corresponding to 
these two tensors in the partition has no influence on the product.  In particular, the outer power of a tensor $\A$ along an unordered partition $\gamma=\set{S_{1}, S_2, \cdots, S_k}$
is defined as     
\beq\label{eq: outprod4unorder}
\A^{\gamma} := \overbrace{[\A, \A, \ldots, \A]}^{k} .
\eeq
\indent  We now define the transposition of tensors. Given a tensor $\A$ of order $p$ and a permutation $\pi$ on set $[p]$. We define $ (\trans_{\pi} \A)_{\mb{i}} :=(\A)_{\mb{i} \circ \pi}$,
or more specifically
\beq\label{eq: deftransport}
  (\trans_{\pi} \A)_{i_1 i_2 \cdots i_p } := A_{i_{\pi(1)}i_{\pi(2)}\cdots i_{\pi(p)}} 
\eeq 
Note that this is a left action on the symmetric group of the index set, i.e.,  $\trans_{\pi\circ \rho} \A =\trans_{\pi}\trans_{\rho} \A$.  Obviously that a tensor is symmetric if it is 
preserved by all transpositions. \\
\indent  The outer product of tensors along a partition can be expressed in terms of the transposition of a canonical outer product, i.e., 
\beq\label{eq: transcanonouterprod}
 \A_1\times_{S_1} \A_2 \times_{S_2}\cdots \times_{S_{k-1}} \A_k = \trans_{\pi} \A_1\times \A_2 \times \cdots \times \A_k,  
\eeq
where $\pi$ is a permutation of $[p]$ mapping the block 
\[ \set{ \sum\limits_{k=1}^{j-1} p_k+1,  \sum\limits_{k=1}^{j-1} p_k +2, \cdots, \sum\limits_{k=1}^{j} p_k },   \]
increasingly onto $S_j$ for all $j\in [k]$.  The requirement of the increasing order on the blocks can be removed if all the $\A_j$'s are symmetric. \\
\indent  The introduction of transposition allows us to formulate the beautiful fact that the outer products of symmetric tensors $\A_{1}, \A_2, \cdots, \A_k$
along all suitable partitions exhaust all transposes of the canonical outer product $\A_{1}\times \A_2\times\cdots \times\A_k$. For the simple case when all $\A_j$'s 
are the same, say $\A:=\A_{1}=\A_2 =\cdots = \A_k$ with order $m$ ($p=mk$), we denote $\A^k:=\overbrace{[\A,\A,\cdots, \A]}^{k}$ for canonical $k$-power of $\A$. \\
\indent  Now let us go back to the case when we are given two matrices, say $A\in \R^{m_1\times n_1},B\in \R^{m_2\times n_2}$. The cross (outer)  product of $A$ and $B$, 
written as $A\times_c B$, is the 4-order tensor with $(A\times_{c} B)_{i_1i_2i_3i_4} = A_{i_1i_3}B_{i_2i_4}$, that is,  $A\times_c B=[A,B]_{\ga}$ with $\ga=\set{\set{1,3},\set{2,4}}$ as a partition of $[4]$. Note that the canonical outer product $A\times B=[A,B]_{\eta}$ where $\eta:=\set{\set{1,2}, \set{3,4}}$. On the other 
hand, the contractive product of a 4-order tensor $\A\in \R^{m\times n\times m\times n}$ with a matrix $P\in \R^{m\times n}$ is defined as  
\[ (\A P)_{ij} = \sum\limits_{\ip,\jp} A_{ij\ip\jp} P_{\ip\jp}, \quad   (P\A)_{ij}= \sum\limits_{\ip,\jp} P_{\ip\jp} A_{\ip\jp i j} \]
The following results can be verified easily (thus proof omitted). 
\begin{cor}\label{cor: tensorprod01} 
\begin{description}
  \item[(1) ]  $(A\times_{(s,t)} I_{n})\times_{(s,t)} I_{n}=n A$ for any 2-subset $\set{s,t}$ of $[4]$.
  \item[(2) ]  $(I_{m}\times I_{n}) A = \tr(A) I_{m}$ for any $A\in \R^{n\times n}$.
  \item[(3) ]  $A\times (I_{m}\times_{c} I_{n}) = (I_{m}\times_{c} I_{n}) A = A$ for any $A\in \R^{m\times n}$.
  \item[(4) ]  $A^{\top}(I_{n}\times_{(2,3)} I_{m}) =A,  (I_{m}\times_{(2,3)} I_{n}) A^{\top} = A$ for any $A\in \R^{m\times n}$.
\end{description}   
\end{cor}
Note also that tensor $I_{m}\times_{c} I_{n}$ can be regarded as the \emph{identity tensor} in the space $\R^{m\times n\times m\times n}$ due to (3) of  Corollary \ref{cor: tensorprod01}. 

\indent  Recall that a commutation matrix $K_{p,q}=(B_{ij})(i\in [p], j\in [q])$ is an $p\times q$ block matrix where each block $B_{ij}\in \R^{q\times p}$ has a unique nonzero 
entry 1 at position $(j,i)$. Thus $K_{2,3}$ is an $6\times 6$ matrix  
\[ 
\left(\begin{array}{cccccc}1 & 0 & 0 & 0 & 0 & 0 \\0 & 0 & 1 & 0 & 0 & 0 \\0 & 0 & 0 & 0 & 1 & 0 \\0 & 1 & 0 & 0 & 0 & 0 \\0 & 0 & 0 & 1 & 0 & 0 \\0 & 0 & 0 & 0 & 0 & 0\end{array}\right)
\]  
In \cite{xh2019},we define the commutation tensor $\K_{n,m}=(K_{ijkl})$ as an $n\times m\times m\times n$ tensor that transforms any matrix $A\in \R^{m\times n}$ into 
its transpose, i.e., $\K_{n,m} A =A^{\top}$.  We also show that $\K_{m,n} =I_{m}\times_{(2,3)} I_{n}$. \\
\indent  Let $X=(x_{ij})\in \R^{m\times n}$ be a random matrix whose entries are independent variables. Let $Y=(y_{ij})\in \R^{p\times q}$ be a matrix each of whose entries 
$y_{ij}$ is a function of $X$.  The derivative $\frac{dY}{dX}$ is interpreted as the 4-order tensor $A=(A_{i_{1}i_{2}i_{3}i_{4}})$ of size $m\times n\times p\times q$ whose entries are defined by 
 \[  A_{i_{1}i_{2}i_{3}i_{4}} = \frac{d Y_{i_{3}i_{4}}}{d X_{i_{1}i_{2}}}.  \]
 
\indent In order to simplify the expressions of high order moments of random matrices, we now use tensors to describe the derivatives of matrices. In the following, we will 
present some known results in tensor forms other than in the conventional matrix versions. The following lemma is the derivative chain rule in the matrix version. \\   
  
\begin{lem}\label{le: matrxcompderiv}
Let $X\in \R^{m_{1}\times n_{1}}, Y\in \R^{m_{2}\times n_{2}}, Z\in \R^{m_{3}\times n_{3}}$, and $Z=Z(Y), Y=Y(X)$.  Then we have  
\beq\label{eq:matrxcompder} 
\frac{dZ}{dX} = \frac{dY}{dX}\times \frac{dZ}{dY}   
\eeq   
\end{lem}

\begin{proof}
We denote $\A=\frac{dY}{dX}, \B=\frac{dZ}{dY}$ and $\C = \frac{dZ}{dX}$.  By definition we have 
\[ \A\in \R^{m_{1}\times n_{1}\times m_{2}\times n_{2}}, \B\in \R^{m_{2}\times n_{2}\times m_{3}\times n_{3}}, \C\in \R^{m_{1}\times n_{1}\times m_{3}\times n_{3}}. \]
Then for any given $(i_{1},i_{2},i_{3},i_{4})\in [m_{1}]\times [n_{1}]\times [m_{3}]\times [n_{3}]$, we have 
\beyy
A_{i_{1}i_{2}i_{3}i_{4}}   &=& (\frac{dZ}{dX})_{i_{1}i_{2}i_{3}i_{4}} = \frac{dz_{i_{3}i_{4}}}{dx_{i_{1}i_{2}}} \\
                                       &=&\sum\limits_{j_{1},j_{2}} \frac{dy_{j_{1}j_{2}}}{dx_{i_{1}i_{2}}}\frac{dz_{i_{3}i_{4}}}{dy_{j_{1}j_{2}}}\\
                                       &=& (\frac{dY}{dX}\times \frac{dZ}{dY})_{i_{1}i_{2}i_{3}i_{4}}
\eeyy
Thus (\ref{eq:matrxcompder}) holds.  
\end{proof}
\indent Lemma \ref{le: matrxcompderiv} can be extended to a more general case:

\begin{lem}\label{le: matrxderivchain23}
\begin{description}
  \item[(1) ]  Let $Z=Z(Y_{1},Y_{2},\ldots, Y_{n})$ be the matrix-valued function of $Y_{1},Y_{2},\ldots, Y_{n}$ where $Y_{k}=Y_{k}(X)$ for all $k\in [n]$. Then
\beq\label{eq:matrxcompder2} 
\frac{dZ}{dX} = \sum\limits_{k=1}^{n} \frac{dY_{k}}{dX}\times \frac{dZ}{dY_{k}}   
\eeq  
  \item[(2) ]  Let $X,Y,Z,U$ be matrix forms of variables and $U=U(Z), Z=Z(Y), Y=Y(X)$. Then we have the chain  
\beq\label{matrxderivchain3}
dU = dX\times (dY/dX)\times (dZ/dY)\times (dU/dZ)  
\eeq
where $dX=(dx_{ij})$.   
\end{description}
\end{lem}

\indent It is easy to verify the results in Lemma \ref{le: matrxderivchain23} so that we omit its proof here. The following results on the matrix derivatives 
are useful and will be used in the next section. 
  
 \begin{thm}\label{th: matrixderivat01}
 Let  $X=(X_{ij})\in \R^{m\times n}$ be a matrix whose elements are independent variables. Then 
 \begin{description}
  \item[(1) ]  $\frac{dX}{dX} = I_{m}\times_{c} I_{n}$.
  \item[(2) ]  $\frac{dX^{\top}}{dX} = I_{m}\times_{(2,3)} I_{n}=\K_{m,n}$.
  \item[(3) ]  $\frac{d(YZ)}{dX} =\frac{dY}{dX}\times_{4} Z + \frac{dZ}{dX}\times_{3} Y^{\top}$. 
  \item[(4) ]  $\frac{dX^{2}}{dX} = I_{n}\times_{c} X  + X^{\top}\times_{c} I_{n}$ when $m=n$.   
  \item[(5) ]  $\frac{dX^{k}}{dX} = \sum\limits_{p=0}^{k-1} \left[(X^{\top})^{p} \times_{c} X^{k-1-p}\right] $ where $X\in \R^{n\times n}$. 
\end{description}   
 \end{thm} 
 
 \begin{proof}
 Let $\A=\frac{dX}{dX}$.  Then we have by definition and the independency of the elements of $X$ that 
 \[ A_{i_{1}i_{2}i_{3}i_{4}} = \frac{d x_{i_{3}i_{4}}}{d x_{i_{1}i_{2}}} = \delta_{i_{1}i_{3}} \delta_{i_{2}i_{4}} = (I_{m}\times_{c} I_{n})_{i_{1}i_{2}i_{3}i_{4}} \]
Thus (1) is proved. Similarly we can prove (2) by noticing that 
\[ \left[\frac{dX^{\top}}{dX}\right]_{i_{1}i_{2}i_{3}i_{4}} = \delta_{i_{1}i_{4}}\delta_{i_{2}i_{3}} \]
which implies 
\[ \frac{dX^{\top}}{dX} = I_{m}\times_{(2,3)} I_{n} =\K_{m,n}.  \]  
\indent  To prove (3), we let $Y\in \R^{p\times r}, Z\in \R^{r\times q}$. Then $\frac{d(YZ)}{dX}\in \R^{m\times n\times p\times q}$ whose elements are
\beyy  
(\frac{d(YZ)}{dX})_{i_{1}i_{2}i_{3}i_{4}} & = &  \frac{d[(YZ)_{i_{3}i_{4}}]}{dX_{i_{1}i_{2}}} = \frac{d[\sum\limits_{k} y_{i_{3}k} z_{k i_{4}}]}{d x_{i_{1}i_{2}}} = \sum\limits_{k} \frac{d(y_{i_{3}k} z_{k i_{4}})}{d x_{i_{1}i_{2}}} \\
                                                              & = &  \sum\limits_{k} \left[ \frac{d(y_{i_{3}k})}{d x_{i_{1}i_{2}}} z_{k i_{4}} + y_{i_{3}k}\frac{d(z_{k i_{4}})}{dx_{i_{1}i_{2}}} \right] \\
                                                              & = & (\frac{dY}{dX}\times_{4} Z)_{i_{1}i_{2}i_{3}i_{4}} +(\frac{dZ}{dX}\times_{3} Y^{\top})_{i_{1}i_{2}i_{3}i_{4}}
\eeyy   
\indent  To prove (4), we let $X\in \R^{n\times n}$ and take $Y=Z=X$. By (3) and (1), we have 
\beyy
\frac{dX^{2}}{dX}  &=& \frac{dX}{dX}\times_{4} X +X\times_{3} \frac{dX}{dX}\\
                             &=& (I_{n}\times_{c} I_{n})\times_{4} X +X\times_{3} (I_{n}\times_{c} I_{n})\\
                             &=& I_{n}\times_{c} X  + X^{\top}\times_{c} I_{n} 
\eeyy
\indent To prove (5), we use the induction method to $k$. For $k=1$, the result is immediate since both sides of (5) are identical to $I_{n}\times_{c} I_{n}$ by (1). 
The result is also valid for $k=2$ by (4).  Now suppose it is valid for a positive integer $k>2$.  We come to show its validity for $k+1$.  By (3) we have  
\beyy
\frac{dX^{k+1}}{dX}  &=& \frac{dX^{k}}{dX}\times_{4} X +(X^{\top})^{k}\times_{3} \frac{dX}{dX}\\
                                &=& \sum\limits_{p=0}^{k-1} \left[(X^{\top})^{p} \times_{c} X^{k-p}\right] + (I_{n}\times_{c} I_{n})\times_{3} (X^{k})^{\top}\\
                                &=& \sum\limits_{p=0}^{k-1} \left[(X^{\top})^{p} \times_{c} X^{k-p} \right]+(X^{\top})^{k}\times I_{n}\\
                                &=& \sum\limits_{p=0}^{k} \left[ (X^{\top})^{p} \times X^{k-p}\right]
\eeyy
Thus we complete the proof of (5). 
\end{proof}

\begin{cor}\label{co: 03} 
Let $X=(X_{ij})\in \R^{m\times n}$ be a matrix whose elements are independent, and $A\in \R^{p\times m},B\in \R^{n\times q}$ be the constant matrices.  Then 
\begin{description}
  \item[(1) ]  $\frac{d(AXB)}{dX} = A^{\top}\times_{c} B$.
  \item[(2) ]  $\frac{d(\det(X))}{dX} = \det(X) X^{-\top}$.
  \item[(3) ]  $\frac{d(\tr(X))}{dX} = I_{n}$ for $X\in \R^{n\times n}$.
  \item[(4) ]  $\frac{dX^{-1}}{dX} = - X^{-\top} \times X^{-1}$ when $X\in \R^{n\times n}$ is invertible.
\end{description} 
\end{cor}

\begin{proof}
To prove (1), we take $Y=A, Z=XB$. By (4) of Theorem \ref{th: matrixderivat01}, we get 
\beq
\frac{d(AXB)}{dX}=\frac{d(XB)}{dX}\times_{3} A^{\top} = A^{\top}\times_{3} (\frac{d(X)}{dX}\times_{4} B)=A^{\top}\times_{3} [(I_{m}\times I_{n})\times_{4} B]=A^{\top}\times B
\eeq
To prove (2), we denote $A=\frac{d(detX)}{dX} =(A_{ij})$. Then for any given pair $(i, j)\in [m]\times [n]$, we have by the expansion of the determinant
\beyy
A_{ij} &=& \frac{d(\det(X))}{X_{ij}} =\frac{d}{dX_{ij}}(\sum\limits_{k=1}^{n} (-1)^{i+k}X_{ik} \det(X(i|k))) \\
         &=& (-1)^{i+j} \det(X(i|j)) =[\det(X)X^{-1}]_{ji}   
\eeyy
where $X(i|j)$ represents the submatrix of $X$ obtained by the removal of the $i$th row and the $j$th column of $X$.  Thus we have  $\frac{d(\det(X))}{dX} = \det(X) X^{-\top}$.  
Now (3) can be verified by noticing the fact that for all $(i,j)$ 
\[ [\frac{d(\tr(X))}{dX}]_{ij} = \sum\limits_{k=1}^{n} \frac{d(X_{kk})}{dX_{ij}} = \sum\limits_{k=1}^{n}\delta_{ik}\delta_{jk} =(e_{i}, e_{j})=\delta_{ij} \]
where $e_{i}\in \R^{n}$ is the $i$th row of the identity matrix $I_{n}$. Now we prove (4). Using (4) of Theorem \ref{th: matrixderivat01} on the equation $XX^{-1}=I_{n}$ (here $X=(X_{ij})\in \R^{n\times n}$), we have
\[ (\frac{d(X)}{dX})\times_{4} X^{-1} + X\times_{2} \frac{dX^{-1}}{dX} \]
It follows that 
\beyy 
\frac{dX^{-1}}{dX} &=& - X^{-1}\times_{2} [ \frac{d(X)}{dX}\times_{4} X^{-1}] \\
                             &=& - X^{-1}\times_{2} [(I_{n}\times I_{n}) \times_{4} X^{-1}] \\  
                             &=& - X^{-1}\times_{2} [(I_{n}\times X^{-1}] \\ 
                             &=& - X^{-\top}\times X^{-1} 
\eeyy
\end{proof}

\indent  Given any two 4-order tensors, say, 
\[ \A=(A_{i_{1}i_{2}j_{1}j_{2}})\in \R^{m_{1}\times m_{2}\times n_{1}\times n_{2}}, \B=(B_{i_{1}i_{2}j_{1}j_{2}})\in \R^{n_{1}\times n_{2}\times q_{1}\times q_{2}}. \]
The product of $\A,\B$, denoted by $\A\B$, is referred to as the 4-order tensor of size $m_{1}\times m_{2}\times q_{1}\times q_{2}$ whose entries are defined by 
\[ (\A\B)_{i_{1}i_{2}j_{1}j_{2}} = \sum\limits_{k_{1}, k_{2}} A_{i_{1}i_{2}k_{1}k_{2}}B_{k_{1}j_{2}j_{1}j_{2}} \]
This definition can also be carried over along other pair of directions. We will not go into detail at this point in this paper, and want to point out that all the results 
concerning the tensor forms of the derivatives can be transformed into the conventional matrix forms, which can be achieved by Kronecker product.

\section{On Gaussian matrices}
\setcounter{equation}{0} 
 
 In this section, we introduce and study the random matrices with Gaussian distributions and investigate the tensor products of such matrices. We denote $\norm{\cdot}$ for 
 the Euclidean norm and $S^{k-1}:=\set{s\in \R^{k}: \norm{s}=1}$ for the \emph{unit sphere} in $\R^{k}$ for any positive integer $k>1$.  Let $m, n>1$ be two 
 positive integers. Then $\al\in S^{m-1}, \be\in S^{n-1}$ implies $\al\otimes \be\in S^{mn-1}$ since $\norm{\al \otimes \be}=\norm{\al}\cdot{} \norm{\be}=1$. \\
 \indent Let $\bX=(x_{ij})\in \R^{m\times n}$ be a random matrix. The characteristic function (CF) of $\bX$ is defined by 
\[ \phi_{\bX}(T) = E[exp(\imath \tr(T^{\p}\bX))], \quad   \forall T\in \R^{m\times n} \] 
Note that $\phi_{\bX}(T) = \phi_{\bx}(\bft)$ where $\bx=\vecc(\bX)$ and $\bft=\vecc(T)$ are respectively the vectorization of $\bX$ and $T$. While vectorization allows us to 
treat all derivative (and thus the high order moments) of random matrices, it also pose a big challenge for identifying the $k$-moment corresponding to each coordinate of 
$\bX$. We introduce the tensor expression for all these basic terminology thereafter.   
\indent  For our convenience, we denote by $X_{i\cdot }$ ($X_{\cdot  j}$) the $i$th row (resp. $j$th column) of a random matrix $\bX$. $\bX$ is called a 
\emph{standard normally distributed} (\emph{SND}) or a \emph{SND} matrix if   
\begin{description}
  \item[(1)]  $X_{i\cdot }$'s i.i.d. with $X_{i\cdot }\sim \normal_{n}(0, I_{n})$;  
  \item[(2)]  $X_{\cdot j}$'s i.i.d. with $X_{\cdot  j}\sim \normal_{m}(0, I_{m})$.
\end{description}
that is, all the rows (and columns) of $X$ are i.i.d. with standard normal distribution. This is denoted by $X\sim \N_{m,n}(0,I_{m}, I_{n})$. \\ 
\indent  The following lemma concerning a necessary and sufficient condition for a SND random vector will be frequently used in the paper.  
\begin{lem}\label{lem: vecsnd}
Let $\bx\in \R^{n}$ be a random vector. Then $\bx\sim \normal_{n}(0, I_{n})$ if and only if $\al^{\top} \bx\sim \normal(0,1)$ for all unit vectors $\al\in S^{n-1}$. 
\end{lem}
\indent  Lemma \ref{lem: vecsnd} is immediate from the fact (see e.g. \cite{KR2005}) that $\bx\sim \normal_{n}(\mu, \Sig)$ if and only if  
\[ \al^{\top} \bx\sim \normal(\al^{\top}\mu, \al^{\top}\Sig \al), \quad \forall \al\in S^{n-1}.\]   
\indent  The following lemma presents equivalent conditions for a random matrix to be SND:
\begin{lem}\label{le: cond4sndm}
Let $X=(x_{ij})\in\R^{m\times n}$ be a random matrix. The following conditions are equivalent:
 \begin{description}
  \item[(1) ]  $X\sim \normal_{m,n}(0, I_{m}, I_{n})$.
  \item[(2) ]  $\vecc(X)\sim \normal_{mn}(0, I_{mn})$.
  \item[(3) ]  All $x_{ij}$ are i.i.d. with $x_{ij}\sim \normal(0,1)$. 
  \item[(4) ]  $\al^{\top} X\be \sim \normal (0, 1), \quad \forall\al\in S^{m-1}, \be\in \R^{n-1}$.
\end{description} 
\end{lem}
\begin{proof}
The equivalence of (2) and (3) is obvious. We now show $(1) \ifff (2) \ifff (4)$.  To show $(2)\implies (1)$, we denote $\bx:=\vecc(X)\in \R^{mn}$ and suppose that  
$\bx\sim\normal_{mn}(0, I_{mn})$. Then $\cov(X_{\cdot  i}, X_{\cdot , j}) =\Sig_{ij} =0$ for all distinct $i,j\in [n]$.  So the columns of $X$ are independent. Furthermore, 
we have by Lemma \ref{lem: vecsnd} that 
\beq\label{eq: prf01thm1}
\al^{\top} \bx \sim \normal (0, 1),\quad  \forall \al\in S^{mn-1}
\eeq
Now set $\al=e_{j}\otimes \be\in \R^{mn}$ ($\forall j\in [n]$) where $e_{j}\in \R^{n}$ is the $j$th coordinate vector of $\R^{n}$ and $\be\in S^{m-1}$. Then 
$\al\in S^{mn-1}$ and by (\ref{eq: prf01thm1}) we have 
\[ 
\be^{\top} X_{\cdot  j} =\be^{\top} X e_{j} =(e_{j}^{\top}\otimes \be^{\top}) \bx =(e_{j} \otimes \be)^{\top} \bx = \al^{\top}\bx\sim \normal(0,1)
\]
It follows by Lemma \ref{lem: vecsnd} that $X_{\cdot  j}\sim \normal_{m} (0, I_{m})$ for all $j\in [n]$. Consequently (2) implies (1) by definition. \\  
\indent $(2)\implies (4)$:  Denote $\ga:=\be\otimes \al$ for any given $\al\in S^{m-1}, \be\in S^{n-1}$. Then $\ga\in S^{mn-1}$. Since 
$\vecc(X)\sim \normal_{mn}(0, I_{mn})$, we have, by Lemma \label{lem: vecsnd}, that 
$\al^{\top}X\be =(\be^{\top}\otimes \al^{\top}) \vecc(X) =\ga^{\top} \vecc(X)\sim \normal(0,1)$, which proves (4). \\
\indent  To show $(4)\implies (2)$, we let $\ga\in S^{mn-1}$. Then there is a unique matrix $A=(a_{ij})\in \R^{m\times n}$ such that 
$\ga = \vecc(A)$ and $\norm{A}_{F}^{2} =\sum_{i,j}a_{ij}^{2} =\norm{\al}^{2}=1$ ($\norm{A}_{F}$ denotes the Frobenius norm of matrix $A$). 
Since $\norm{\vecc(A)}_{2} = \norm{A}_{F} =1$, $\vecc(A)\in S^{mn-1}$.  By (4), we have $\ga^{\top}\vecc(X)  = \vecc(A)^{\top} \vecc(X)\sim \normal (0,1)$.
Consequently we get $X\sim \normal_{mn}(0, I_{mn})$ by Lemma \ref{lem: vecsnd}. 
\end{proof} 
\indent The density function and the characteristic function of a SND random matrix \cite{BB1999, KR2005} can also be obtained by Lemma \ref{le: cond4sndm}. 
\begin{prop}\label{pp: snmfphi} 
Let $X\sim \normal_{m,n}(0,I_{m}, I_{n})$. Then 
\begin{description}
 \item[(1)]  $f_{X}(T)=(2\pi)^{-mn/2} \exp\set{-\frac{1}{2}\tr(T^{\top}T)}$ where $T\in \R^{m\times n}$.
 \item[(2)]  $\phi_{X}(T)=\exp\set{-\frac{1}{2}\tr(T^{\top}T)}$ where $T\in \R^{m\times n}$.
\end{description}
\end{prop}

\indent  Let $M=(m_{ij})\in\R^{n_{1}\times n_{2}}$, and $\Sig_{k}=(\si_{ij}^{(k)})\in\R^{n_{k}\times n_{k}}$ be positive definite for $k=1,2$.  A random matrix 
$X=(X_{ij})\in \R^{n_{1}\times n_{2}}$ is called a \emph{Gaussian matrix}\footnote{We do not use the term normal matrix since it is referred to a matrix satisfying $XX^{\top} =X^{\top}X$} with parameters $(M, \Sig_{1}, \Sig_{2})$, written as $X\sim \N_{m,n}(M,\Sig_{1}, \Sig_{2})$, if  
\begin{description}
\item[(a)]  Each row $X_{i \cdot }$ follows a Gaussian distribution with    
\beq\label{eq: defnm01} 
X_{i\cdot}\sim \normal_{n}(M_{i\cdot},\si_{ii}^{(1)}\Sig_{2}), \quad  \forall i\in [m],
\eeq 
\item[(b)] Each column vector $X_{\cdot j}$ follows a Gaussian distribution with  
 \beq\label{eq: defnm02} 
X_{\cdot j}\sim \normal_{m}(M_{\cdot j},\si_{jj}^{(2)}\Sig_{1}), \quad  \forall j\in [n]
\eeq 
\end{description}
Such a random matrix $X$ is called a \emph{Gaussian matrix}. It follows that the vectorization of a Gaussian matrix $X$ is a Gaussian vector, i.e.,   
\beq\label{eq: vec2nm}
\vecc(X) \sim \normal_{mn}(\vecc(M), \Sig_{2}\otimes\Sig_{1})
\eeq
A Gaussian vector cannot be shaped into a Gaussian matrix if its covariance matrix possesses no Kronecker decomposition of two PSD matrices. For any two random 
matrices (vectors) $X,Y$ of the same size, we denote $X=Y$ if their distributions are identical. The following statement, which can be found in \cite{KR2005}, shows that 
an affine transformation preserves the Gaussian distribution.  
 \begin{lem}\label{le: gdtran}
Let $X\sim \N_{n_{1},n_{2}}(\mu, \Sig_{1}, \Sig_{2})$ and $Y=B_{1}XB_{2}^{\top}+C$ with $B_{i}\in\R^{m_{i}\times n_{i}}$ ($i=1,2$) being constant matrices.  Then
\[ Y\sim \N_{m_{1}, m_{2}}(C+B_{1}\mu B_{2}^{\top}, B_{1}\Sig_{1}B_{1}^{\top}, B_{2}\Sig_{2}B_{2}^{\top})  \]   
\end{lem}
\indent  The following statement can be regarded as an alternative definition of a Gaussian matrix.\\      
\begin{lem}\label{le: gd2sgd}
Let $X=(x_{ij})\in \R^{n_{1}\times n_{2}}$ be a random matrix and $\Sig_{i}=A_{i}A_{i}^{\top}$ with each $A_{i}\in\R^{n_{i}\times n_{i}}$ nonsingluar ($i=1,2$). Then  
$X\sim \N_{n_{1},n_{2}}(M, \Sig_{1}, \Sig_{2})$ if and only if there exist a SND random matrix $Z\in \R^{p\times q}$ such that 
\beq\label{eq:snd2n}
X=A_{1}ZA_{2}^{\top}+M
\eeq 
where $M\in \R^{m\times n}$ is a constant matrix.   
\end{lem}
\indent The proof of Lemma \ref{le: gd2sgd} can be found in \cite{BB1999}.  Now we let $X\sim \N_{n_{1}, n_{2}}(M,\Sig_{1}, \Sig_{2})$ be a Gaussian matrix where 
$M=(m_{ij})\in\R^{n_{1}\times n_{2}}$ and $\Sig_{k}=(\si_{ij}^{(k)})\in\R^{n_{k}\times n_{k}}$ ($k=1,2$) being positive definite. Write  
\beq\label{eq:omegT}
\om_{T}=\tr[(T-M)^{\top}\Sig_{1}^{-1}(T-M)\Sig_{2}^{-1}]
\eeq  
where $T\in \R^{n_{1}\times n_{2}}$ is arbitrary.  The characteristic function (CF) of $X$ is defined as   
\beq\label{eq:cf4X} 
\phi_{X}(T) := E[\exp(\imath\seq{X,T})], \quad  T\in \R^{n_{1}\times n_{2}}
\eeq 
We have 
\begin{cor}\label{co: matrxdensitychar}
Let $X\in \R^{n_{1}\times n_{2}}$ be a random matrix. Then the density and the characteristic function of $X$ are respectively 
\beq\label{eq: df4X}
 f_{X} (T) = (2\pi)^{-n_{1}n_{2}/2} (\det(\Sig_{1}))^{-n_{2}/2} (\det(\Sig_{2}))^{-n_{1}/2} \exp\set{-\frac{1}{2}\om_{T} }
  \eeq 
and   
  \beq\label{eq:cf4X2}
 \phi_{X} (T) = \exp\set{\imath \tr(T^{\top} M)-\frac{1}{2} \tr(T^{\top}\Sig_{1} T\Sig_{2})}
  \eeq 
where $T$ takes values in $\R^{n_{1}\times n_{2}}$. 
\end{cor}

\indent Lemma \ref{le: gdtran} can be used to justify the definition of Gaussian matrices if we take Lemma \ref{le: gd2sgd} as the original one. 
Let $X\sim \N_{n_{1},n_{2}}(\mu,\Sig_{1}, \Sig_{2})$, $A=I_{m}$ and $B=e_{j}$ ($\forall j\in [n]$) is the $j$th coordinate vector of $\R^{n}$. Then we have  
\[ AXB=X_{\cdot j},  AMB=\mu_{\cdot j}, A\Sig_{1}A^{\top}=\Sig_{1}, B\Sig_{2}B^{\top}=\si_{jj}^{2}, \]
Thus $\bx_{j}\sim \normal_{m,1} (\mu_{\cdot j}, \Sig_{1}, \si_{jj}^{2})$ by Lemma \ref{le: gdtran}, which is equivalent to (\ref{eq: defnm02}). This argument also applies 
to prove (\ref{eq: defnm01}). Furthermore, $x_{ij}\sim \normal(\mu_{ij}, (\si^{(1)}_{ii}\si^{(2)}_{jj})^{2})$ for all $i\in [n_{1}],  j\in [n_{2}]$. \\ 
\indent For any matrix $A\in \R^{m\times n}$, we use $A[S_{1}|S_{2}]$ to denote the submatrix of $A$ whose entries $a_{ij}$'s are confined in $i\in S_{1}, j\in S_{2}$ 
where $\emptyset\neq S_{1}\subset [m], \emptyset \neq S_{2}\subset [n]$. This is denoted by $A[S]$ when $S_{1}=S_{2}=S$.  It follows from Lemma \ref{le: gdtran} 
that any submatrix of a Gaussian matrix is also Gaussian. \\ 
\begin{cor}\label{co: subgaussm}
Let $X\sim \N_{n_{1},n_{2}}(\mu,\Sig_{1}, \Sig_{2})$. and $\empty\neq S_{i}\subset [n_{i}]$ with cardinality $\abs{S_{i}}=r_{i}$ for $i=1,2$. Then  
\beq\label{subgauss} 
X[S_{1}|S_{2}] \sim \N_{r_{1}, r_{2}} (\mu[S_{1}|S_{2}], \Sig_{1}[S_{1}], \Sig_{2}[S_{2}])  
\eeq   
\end{cor}

\begin{proof}
We may assume that $S_{1}=\set{i_{1}<i_{2}<\ldots <i_{r_{1}}}, S_{2}=\set{j_{1} < j_{2}< \ldots < j_{r_{2}}}$, and for $i=1,2$, choose matrix  
\[ P_{1}^{\top} =[e_{i_{1}}, e_{i_{2}},\ldots, e_{i_{r_{1}}}], \quad  P_{2}^{\top} =[f_{j_{1}}, f_{j_{2}},\ldots, f_{j_{r_{2}}}] \]
where $e_{k}\in \R^{n_{1}}$ is the $k$th coordinate (column) vector in $\R^{n_{1}}$, and $f_{k}\in \R^{n_{2}}$ is the $k$th coordinate (column) vector in $\R^{n_{2}}$, thus 
we have $P_{i}\in \R^{r_{i}\times n_{i}}$ for $i=1,2$. Since $X[S_{1}|S_{2}] =P_{1}XP_{2}^{\top}$, we have 
\[  X[S_{1}|S_{2}] \sim \N_{r_{1}, r_{2}} (P_{1} \mu P_{2}^{\top}, P_{1}\Sig_{1}P_{1}^{\top}, P_{2}\Sig_{2}P_{2}^{\top} )  \]
Then (\ref{subgauss}) follows by noticing that $\mu[S_{1}|S_{2}]=P_{1}\mu P_{2}^{\top}$ and  $\Sig_{i}[S_{i}]=P_{i}\Sig_{i}P_{i}^{\top}$ ($i=1,2$). 
\end{proof} 
\indent For a random vector $\bx\sim \normal_{n}(\mu, \Sig)$, we have $\bx=\mu +A\by$ with $\mu\in \R^{n}, A\in \R^{n\times n}$ satisfying
$AA^{\top}=\Sig$ and $\by\sim \normal_{n}(0, I_{n})$.  It follows that $m_{1}[\bx]=\mu$ and 
 \beyy  
 m_{2}[\bx] &=& E[\bx\bx^{\top}]= E[(\mu +A\by)(\mu +A\by)^{\top}] \\
                  &=& \mu\mu^{\top} +AE[\by\by^{\p}] A^{\top}\\
                  &=& \mu\mu^{\top} +AA^{\top} = \mu\mu^{\top} +\Sig
 \eeyy
since $E[\by\by^{\top}]=m_{2}[\by]=I_{n}$.  The $k$-moment of a random matrix $X\in \R^{m\times n}$ is defined as the $2k$-order tensor 
$E[X^{(k)}]$ which is of size $m^{[k]}\times n^{[k]}$.  Write $m_{k}[X]=(\mu_{i_{1}i_{2}\ldots i_{k}j_{1}j_{2}\ldots j_{k}})$. By definition 
\[ \mu_{i_{1}i_{2}\ldots i_{k}j_{1}j_{2}\ldots j_{k}} = E[x_{i_{1}j_{1}}x_{i_{2}j_{2}}\ldots x_{i_{k}j_{k}}]  \]
 
\begin{lem}\label{le: snm}
Let $Y\in \R^{m\times n}$ be a standard normal matrix (SNM), i.e., $Y\sim\normal_{m,n}(0, I_{m}, I_{n})$. Then 
\beq\label{eq: snmm2} 
m_{2}[Y]=E[Y\times Y] =I_{m}\times_{c} I_{n}
\eeq
\end{lem}

\begin{proof}
Denote $\by:=\vecc(Y)$. Then $\by\sim \normal_{mn}(0, I_{mn})$ by the hypothesis.  It follows that $m_{2}[\by]=I_{mn}$ by Lemma \ref{le: cond4sndm}.  Now 
let $Z = Y\times Y$ and $M^{(2)}=m_{2}[Y]$.  Then  
\beq\label{eq:prf4snm}
M^{(2)}_{i_{1}i_{2}i_{3}i_{4}} = E[y_{i_{1}i_{2}}y_{i_{3}i_{4}}] = \cov(y_{i_{1}i_{2}}, y_{i_{3}i_{4}}) =\delta_{i_{1}i_{3}} \delta_{i_{2}i_{4}} =(I_{m}\times_{c} I_{n})_{i_{1}i_{2}i_{3}i_{4}} 
\eeq
for any $(i_{1},i_{2},i_{3},i_{4})$.  So (\ref{eq: snmm2}) holds. The third equality in (\ref{eq:prf4snm}) is due to Lemma \ref{le: cond4sndm}.  
\end{proof}
 
\indent Let $X$ be a random matrix following MND $X\sim\normal_{n_{1},n_{2}}(\mu, \Sig_{1}, \Sig_{2})$ with $\mu\in \R^{n_{1}\times n_{2}}$ and 
$\Sig_{i}\in \R^{n_{i}\times n_{i}}$ is PSD, $i=1,2$.  Write $\Sig_{i}=A_{i}A_{i}^{\top}$ with $A_{i}\in\R^{n_{i}\times n_{i}}$ being nonsingular ($i=1,2$).  
Denote $Z=A_{1}^{-1}(X-\mu) A_{2}^{-\top}$.  Then $Z\in \R^{n_{1}\times n_{2}}$ is a SND matrix and we have 
\beq\label{eq: fromZ2X}
 X = \mu + A_{1}ZA_{2}^{\top}
\eeq 
It follows that $m_{1}[X]=E[X] = \mu \in \R^{n_{1}\times n_{2}}$  and $m_{2}[X]=E[X\times X]$ whose entries are defined by  
\beq\label{eq: m2entry} 
m^{(2)}_{i_{1}i_{2}i_{3}i_{4}} = E[x_{i_{1}i_{2}} x_{i_{3}i_{4}}]=\cov(X_{i_{1}i_{2}}, X_{i_{3}i_{4}})
\eeq
Thus $m_{2}[X] \in \R^{n_{1}\times n_{2}\times n_{1}\times n_{2}}$.  An \emph{$k$-moment} of $X$ is defined as an $2k$-order tensor 
$M_{k}[X] := E[X^{[k]}]$ (with size $(n_{1}\times n_{2})^{[k]}$).  Then each entry of $M_{k}[X]$ can be described as 
\[ (M_{k}[X])_{i_{1}i_{2}\ldots i_{k}j_{1}j_{2}\ldots j_{k}} = E[X_{i_{1}j_{1}}X_{i_{2}j_{2}}\ldots X_{i_{k}j_{k}}]  \]
For any matrices $A,B,C,D$ and non-overlapped subset $\set{s_{i},t_{i}}\subset [8]$ with $s_{i} < t_{i}$. The tensor 
\[ \T = (T_{i_{1}i_{2}\ldots i_{8}}) \equiv A\times_{(s_{2},t_{2})} B\times_{(s_{3},t_{3})} C\times_{(s_{4},t_{4})} D \]
yields an 8-order tensor whose entries are defined by 
\[ T_{i_{1}i_{2}\ldots i_{8}} = A_{i_{s_{1}}i_{t_{1}}}B_{i_{s_{2}}i_{t_{2}}}C_{i_{s_{3}}i_{t_{3}}}D_{i_{s_{4}}i_{t_{4}}} \]
where $\set{s_{1},t_{1}} = (\cup_{k=2}^{4} \set{s_{k}, t_{k}})^{c}$. \\ 

\indent Let $\A, \B$ be tensors of order $p$ and $q$ respectively.  The \emph{tensor product} $\C:=\A\times \B$ is an $(p+q)$-order tensor whose components 
are defined by 
\[ C_{i_{1}i_{2}\ldots i_{p}j_{1}j_{2}\ldots j_{q}} = A_{i_{1}i_{2}\ldots i_{p}} B_{j_{1}j_{2}\ldots j_{q}} \]
For $p=q$, we can also define the \emph{cross tensor product} of $\A,\B$, as the $2p$-order tensor $\D=\A\times_{c} \B$ defined by 
\[ D_{i_{1}j_{1} i_{2}j_{2}\ldots i_{p} j_{p}} = A_{i_{1}i_{2}\ldots i_{p}} B_{j_{1}j_{2}\ldots j_{p}} \]   
Note that when $A\in \R^{m\times n}, B\in \R^{m\times n}$, we have 
\[ A\times B=A\times_{(3,4)} B\in \R^{m\times n\times m\times n}, \quad   A\times_{c} B=A\times_{(2,4)} B\in \R^{m\times m\times n\times n}. \]
In the following when we write $A\times B$ we usually mean $A\times_{(3,4)} B$, the tensor product of $A$ and $B$, if not mentioned the other way.   
 
\begin{lem}\label{le: derA}
Let $A(T):=\imath \mu - \Sig_{1} T \Sig_{2}$. Then we have  
\begin{description}
  \item[(1) ]  $\frac{dA}{dT} = - \Sig_{1}\times \Sig_{2}$.  
  \item[(2) ]  $\frac{d(A\times_{c} A)}{dT} = \frac{dA}{dT}\times_{(3,6)} A + \frac{dA}{dT}\times_{c} A$. 
\end{description}
\end{lem}    

\begin{proof} 
We write $\B=(B_{i_{1}i_{2}j_{1}j_{2}})= \frac{dA(T)}{dT}$ i.e., the derivative of $A(T)$ w.r.t. $T$, which, according to the definition, is of size 
$n_{1}\times n_{1}\times n_{2}\times n_{2}$. Then
\beq
B_{i_{1}i_{2}j_{1}j_{2}} =\frac{dA_{i_{2}j_{2}}}{dT_{i_{1}j_{1}}} = -\frac{d}{dT_{i_{1}j_{1}}}\sum\limits_{k,l} \si_{i_{2}k}^{(1)}\si_{l j_{2}}^{(2)} T_{kl}
                                     = - \sum\limits_{k,l} \si_{i_{2}k}^{(1)}\si_{l j_{2}}^{(2)} \delta_{i_{1}k} \delta_{j_{1}l} =- \si_{i_{2}i_{1}}^{(1)}\si_{j_{1} j_{2}}^{(2)}
\eeq
It follows that $\B=\frac{dA}{dT} = -\Sig_{1}\times \Sig_{2}$.\\
\indent  To prove (2), we first note that $A\times_{c} A\in \R^{n_{1}\times n_{1}\times n_{2}\times n_{2}}$.  Thus $\frac{d(A\times_{c} A)}{dT}$ is an 6-order tensor. Denote 
$\C=(C_{i_{1}i_{2}i_{3}j_{1}j_{2}j_{3}})= \frac{d(A\times_{c} A)}{dT}$. Then $\C$ is of size $n_{1}^{[3]}\times n_{2}^{[3]}\equiv n_{1}\times n_{1}\times n_{1}\times n_{2}\times n_{2}\times n_{2}$, and    
\beq  
C_{i_{1}i_{2}i_{3}j_{1}j_{2}j_{3}} = \frac{d[A_{i_{2}j_{2}}A_{i_{3}j_{3}}]}{dT_{i_{1}j_{1}}} =
 (\frac{dA}{dT})_{i_{1}i_{2}j_{1}j_{2}} A_{i_{3}j_{3}} + (\frac{dA}{dT})_{i_{1}i_{3}j_{1}j_{3}} A_{i_{2}j_{2}} 
\eeq 
for all possible $i_{k},j_{k}$, which completes the proof of (2).  
\end{proof}

\indent Let  $\bX\sim \normal_{n_{1},n_{2}}(\mu, \Sig_{1}, \Sig_{2})$ and let $\phi:=\phi_{X}(T)$ be its characteristic function. For our convenience, 
we denote $\phi^{\p}$ ($A^{\p}$) for the first order derivative of $\phi$ ($A(T)$) w.r.t. $T$, $\phi^{\p\p}$ ($A^{\p\p}$) for the second order derivative of $\phi$ ($A(T)$) 
w.r.t. $T$, and $\phi^{(k)}$ ($A^{(k)}$) for the $k$-order derivative of $\phi$ ($A(T)$)w.r.t. $T$.  By Lemma \ref{le: derA}, we can characterize the derivatives of the 
characteristic function of a Gaussian matrix $X$, as in the following: 
\begin{thm}\label{th: phi4nm}
Let  $\bX\sim \normal_{n_{1},n_{2}}(\mu, \Sig_{1}, \Sig_{2})$, $\phi:=\phi_{X}(T)$ be its characteristic function and $A=A(T)$ be defined as above.  Then 
\begin{description}
  \item[(1)]   $\phi^{\p}  = \phi (\imath\mu -\Sig_{1} T \Sig_{2}) =\phi A$. 
  \item[(2)]   $\phi^{\p\p} = -\phi \left[ \mu^{[2]} +\Sig_{1}\times \Sig_{2} -\Sig_{1}^{[2]}T^{[2]}\Sig_{2}^{[2]} + 
                   \imath (I_{n_{1}}\times \Sig_{1})(\mu\times T+T\times \mu)(I_{n_{2}}\times \Sig_{2}) \right]$ 
  \item[(3)]   $\phi^{(3)} = A\times_{(3,6)} \phi^{\p\p} + \phi (A\times A)^{\p}$.
  \item[(4)]   $\phi^{(k+1)} = \sum\limits_{i+j=k} A^{(i)}\times_{(3,6)}\phi^{(j)} + \sum\limits_{i+j=k-1} \phi^{(i)}\times_{(2,5)} A^{(j)}$.
\end{description}
\end{thm}

\begin{proof}
In order to prove (1), we denote $ f=\tr(\Sig_{1}T\Sig_{2}T^{\p})$. Then $\frac{df}{dT}\in \R^{n_{1}\times n_{2}}$.  Since 
\[ \tr(\Sig_{1}T\Sig_{2}T^{\p}) =(\Sig_{1}T, (T\Sig_{2})^{\p})=\sum\limits_{i,j,k,l} \si^{(1)}_{ij}\si^{(2)}_{lk}t_{jk}t_{il}, \]
It follows that for any pair $(u,v)$ where $u\in [n_{1}], v\in [n_{2}]$, we have 
\beyy
(df/dT)_{uv} = df/dt_{uv} &=& \sum\limits_{i,j,k,l} \si^{(1)}_{ij}\si^{(2)}_{lk}(\delta_{ju}\delta_{kv}t_{il} + \delta_{iu}\delta_{lv}t_{jk} ) \\
            &=& \sum\limits_{i,l} \si^{(1)}_{iu}\si^{(2)}_{lv}t_{il} + \sum\limits_{j,k} \si^{(1)}_{uj}\si^{(2)}_{vk}t_{jk}  \\
            &=& 2(\Sig_{1}T\Sig_{2})_{uv}
\eeyy
Thus we have 
\beq\label{eq: trderiva}
\frac{d(\tr(\Sig_{1}T\Sig_{2}T^{\p}))}{dT} = 2\Sig_{1}T\Sig_{2}
\eeq
It follows that $\phi^{\p}(T) := \frac{d\phi}{dT} = \phi (\imath\mu -\Sig_{1} T \Sig_{2})$ due to (\ref{eq:cf4X}).  Thus (1) is proved. \\
\indent  To prove (2), we denote $A:=A(T)=\imath\mu -\Sig_{1} T \Sig_{2}$ as in Lemma \ref{le: derA}.  Then again by Lemma \ref{le: derA} we get  
\beyy 
    \frac{d^{2}\phi}{dT^{2}} &=& \frac{d}{dT} (\phi^{\p}) = \frac{d}{dT} (\phi A) = \frac{d\phi}{dT}\times_{(2,4)} A + \phi (\frac{dA}{dT})\\
                                         &=& -\phi [\mu\times_{(2,4)}\mu+\Sig_{1}\times_{(2,4)}\Sig_{2} -U\times_{(2,4)} U+\imath (\mu\times_{(2,4)}U + U\times_{(2,4)}\mu)]
\eeyy
where $U=\Sig_{1}T\Sig_{2}$.  By (4) of Lemma \ref{prop:assoc4tensorprod}, $U\times U=(\Sig_{1}\times \Sig_{1})(T\times T)(\Sig_{1}\times \Sig_{1})$ and thus (2)
holds. \par
\indent   Now (3) can be verified by using Lemma \ref{le: derA}, and (4) is also immediate if we use the induction approach to take care of it. 
\end{proof}

\begin{cor}\label{co: m4gvec}
Let  $\bX\sim \normal_{n_{1},n_{2}}(0, \Sig_{1}, \Sig_{2})$.  Then 
\begin{description}
  \item[(1) ]   $m_{2}[\bX] =\Sig_{1}\times\Sig_{2}$ . 
  \item[(2) ]   $m_{k}[\bX] = 0$ for all odd $k$.
  \item[(3) ]   $m_{4}[\bX] = \Sig_{1}\times_{(2,4)} \Sig_{2}\times_{(5,7)} \Sig_{1}\times_{(6,8)} \Sig_{2} +\Sig_{1}\times_{(2,4)} \Sig_{2}\times_{(3,7)} \Sig_{1}\times_{(6,8)} \Sig_{2}+\Sig_{1}\times_{(2,6)} \Sig_{2}\times_{(5,7)} \Sig_{1}\times_{(4,8)} \Sig_{2} $
\end{description}
\end{cor}
\begin{proof}
(1).  By definition we have 
\[ m_{2}[\bX] =\frac{1}{\imath^{2}} \phi^{\p\p}(T) _{T=0} = -\phi^{\p\p}(T) _{T=0}. \]
The result is followed by (2) of Theorem \ref{th: phi4nm}.\par
(2).  It is obvious that  $m_{1}[\bX] = 0$.  By the hypothesis, we have $A(0)=0, A^{\p}(0)=-\Sig_{1}\times \Sig_{2}$.  Thus we have 
\[ (A\times A)^{\p}\mid_{T=0} = A^{\p}(0)\times_{(3,6)} A(0) + A^{\p}(0)\times_{c} A(0) =0, \]
We now use the induction to $k$ to prove (2).  By Theorem \ref{th: phi4nm}, we have   
\beyy 
 m_{3}[\bX]  &=& \frac{1}{\imath^{3}} \phi^{(3)}(T) _{T=0} \\
                    &=& -\imath \left[ A(0)\times_{(3,6)} \phi^{\p\p} (0) +\phi (A\times A)^{\p}\right]|_{T=0}\\
                    &=& -\imath A(0)\times_{(3,6)} \phi^{\p\p} (0) =0   
\eeyy
Now we assume the result holds for an odd number $k$. Then by (4) of Theorem \ref{th: phi4nm}, we have
\beyy 
m_{k+2}  &=& \frac{1}{\imath^{k+2}} \phi^{(k+2)}\mid_{T=0} \\
               &=& \frac{1}{\imath^{k+2}} \left( A\times_{(3,6)}\phi^{(k+1)} + A^{\p}\times_{(3,6)}\phi^{(k)} + A\times_{(3,6)}\phi^{(k+1)} \right)\mid_{T=0}\\
               &=& \frac{1}{\imath^{k+2}} \left( A(0)\times_{(3,6)}\phi^{(k+1)}(0) + A^{\p}(0)\times_{(3,6)}\phi^{(k)}(0) + A(0)\times_{(3,6)}\phi^{(k+1)}(0)\right) \\
               &=& 0
\eeyy
since $A(0)=0$ and $\phi^{(k)}(0)$ by the hypothesis. Thus the result is proved. \par 
(3).  This can be shown by using (4) of Theorem \ref{th: phi4nm}. But we can also prove it by comparing the item (iv) in Theorem 2.2.7 (Page 203) in \cite{KR2005}. 
\end{proof}

\section{Random tensors with Gaussian distributions}
\setcounter{equation}{0} 
  
\indent We start this section by considering a 3-order random tensor $\Z\in  \R^{n_{1}\times n_{2}\times n_{3}}$.  The unfolding of $\Z$ along mode-$k$ for any $k\in [3]$ 
is a random matrix $Z[k]\in \R^{n_{k}\times n_{i}n_{j}}$ where $\set{i,j,k}=[3]$. A Gaussian tensor is a random tensor with a Gaussian distribution.  For our convenience, we 
use $Z(:, j, k)$ to denote the $(j,k)$-fibre of $\Z$  (notation borrowed from MATLAB) given $j\in [n_{2}], k\in [n_{3}]$, and  use $A(i,:,k)$ and $A(i,j,:)$ for the similar cases. 
Now we state an equivalent definition for a 3-order SND tensor.       
\begin{definition}\label{def: sndtensor}
A random tensor $\Z\in  \R^{n_{1}\times n_{2}\times n_{3}}$ is said to follow a \emph{standard normal distribution}  (\emph{SND}), denoted
$\Z\sim \normal(0, I_{n_{1}},I_{n_{2}},I_{n_{3}})$, if the following three conditions hold 
\begin{description}
  \item[(1) ]  For all $j\in [n_{2}], k\in [n_{3}]$, $Z(:,j,k)\in \R^{n_{1}}$'s are i.i.d. with $\normal_{n_{1}}(0, I_{n_{1}})$.    
  \item[(2) ]  For all $i\in [n_{1}], k\in [n_{3}]$, $Z(i,:,k)\in \R^{n_{2}}$'s are i.i.d. with $\normal_{n_{2}}(0, I_{n_{2}})$. 
  \item[(3) ]  For all $i\in [n_{1}], j\in  [n_{2}]$, $Z(i,j,:)\in \R^{n_{3}}$'s are i.i.d. with $\normal_{n_{3}}(0, I_{n_{3}})$. 
\end{description} 
\end{definition}  
\indent Similar to the random matrix case, we have 
\begin{thm}\label{th: cond4sndt}
Let $\Z = (Z_{ijk}) \in \R^{n_{1}\times n_{2}\times n_{3}}$ be a random tensor, $n=n_{1}n_{2}n_{3}$ and $m_{l}=n/n_{l}$ for $l=1,2,3$.  The following items are equivalent:  
 \begin{description}
  \item[(1) ]  $\Z \sim \normal (0, I_{n_{1}}, I_{n_{2}}, I_{n_{3}})$.
  \item[(2) ]  $Z[k]\sim \normal (0, I_{n_{k}}, I_{m_{k}})$ for all $k = 1, 2, 3$.
  \item[(3) ]  $\vecc(\Z)\sim \normal_{n}(0, I_{n})$. 
  \item[(4) ]  All the $Z_{ijk}$'s are i.i.d. with $Z_{ijk}\sim \normal(0,1)$. 
  \item[(5) ]  $\Z \times_{1}\al^{(1)}\times_{2}\al^{(2)} \times_{3}\al^{(3)}\sim \normal (0, 1), \forall  \al^{(i)} \in S^{n_{i}-1}, i=1,2,3$.
\end{description} 
\end{thm}

\begin{proof}
We first show that $(1)\implies (2)$. It suffices to show that $Z[1]\sim \normal_{n_{1},m_{1}}(0, I_{n_{1}}, I_{m_{1}})$ since $X[1]\in \R^{n_{1}\times m_{1}}$.  
By definition of a SND tensor, we know that $Z(:,j,k)\sim \normal_{n_{1}}(0, I_{n_{1}})$ for all $j,k$. On the other hand, for any $i\in [n_{1}]$, we have 
$\vecc(Z(i,:,:))\sim \normal_{m_{1}}(0, I_{m_{1}})$ by $Z(i,:,:)\sim \normal_{n_{2}, n_{3}}(0, I_{n_{2}}, I_{n_{3}})$ and Lemma \ref{le: cond4sndm}.  Thus 
$Z[1](i,:)=(\vecc(Z(i,:,:)))^{\top}\sim \normal_{m_{1}}(0, I_{m_{1}})$. The implication $(2)\implies (1)$ is obvious. \\   
\indent The equivalence $(2)\ifff (3)$ is directly from Lemma \ref{le: cond4sndm} and  $(3)\ifff (4)$ is obvious. We now prove $(3)\implies (5)$. Let 
$\al^{(l)}\in S^{n_{l}-1}$ for $l=1,2,3$ and let $\be=\al^{(2)}\otimes \al^{(3)}$. Then we have $\be\in S^{m_{1}-1}$. Furthermore, we have 
\[ \Z\times_{1}\al^{(1)}\times_{2}\al^{(2)} \times_{3}\al^{(3)} =  (\al^{(1)})^{\top}Z[1]\be  \] 
which is SND by Lemma \ref{le: cond4sndm}.  Thus we have $\Z\times_{1}\al^{(1)}\times_{2}\al^{(2)} \times_{3}\al^{(3)} \sim \normal (0, 1)$ for all 
$\al^{(l)} \in S^{n_{l}-1}$ ($l=1,2,3$). Thus (5) holds. \\
\indent  Conversely, we let (5) hold and want to show (2). For any given $\al\in S^{n_{1}-1}$, we denote $A(\al)=\Z \times_{1}\al$. Then 
$A(\al)\in \R^{n_{2}\times n_{3}}$. Furthermore, for any $ \al^{(2)} \in S^{n_{2}-1}, \al^{(3)} \in S^{n_{3}-1}$, we have 
\[ \Z\times_{1}\al^{(1)}\times_{2}\al^{(2)} \times_{3}\al^{(3)} = (\al^{(2)})^{\top} A(\al) \al^{(3)}\in \normal(0,1)   \]
It follows by Lemma \ref{le: cond4sndm} that $A(\al)\in \normal_{n_{2},n_{3}}(0, I_{n_{2}}, I_{n_{3}})$ for every $\al\in S^{n_{1}-1}$.  Specifically, if we take 
$\al=e_{i}\in \R^{n_{1}}$ to be the $i$th coordinate vector in $\R^{n_{1}}$, then we have $A(\al)=Z(i,:,:)$ ($i\in [n_{1}]$).  Hence we have   
\[ Z(i,:,:) \in \normal_{n_{2},n_{3}}(0, I_{n_{2}}, I_{n_{3}}), \forall i\in [n_{1}]. \]
This shows that all the slices of $\Z$ along the mode-1 is a SND matrix. We can also show that all the slices (along the other two directions) are SND matrices by 
employing the same technique. This complete the proof that all the five items are equivalent.  
\end{proof} 

\indent  From Definition \ref{def: sndtensor}, we can see that a hypercubic random tensor $\Z\in \T_{3;n}$ ($\T_{3;n}:=\R^{n\times n\times n}$) is SND if 
$\Z\bx^{3}\sim \normal(0,1)$ for any unit vector $\bx\in \R^{n}$. It is easy to see from Definition \ref{def: sndtensor} that 

\begin{cor}\label{le: sndt}
Let $\Z=(Z_{ijk}) \in \R^{n_{1}\times n_{2}\times n_{3}}$ be a random tensor.  Then $\Z$ is a SND tensor if and only if all $Z_{ijk}$'s are  $iid$ with $Z_{ijk}\sim \normal(0,1)$.  
\end{cor} 

\indent  Similar to the matrix case, we can also define a general Gaussian tensor.  Denote $N:=n_{1}n_{2}\ldots n_{m}, m_{k}:=N/n_{k}$..  
Let $\M\in \R^{n_{1}\times \ldots \times n_{m}}$ be a constant tensor and $\Sig_{k}\in \R^{n_{k}\times n_{k}}$ be a positive semidefinite matrix 
for each $k\in [m]$.  For any random tensor $\Z\in \R^{n_{1}\times \ldots \times n_{m}}$, we denote $Z^{(k)}(:,j)$ for the $j$th column (fiber) vector of 
$Z^{(k)}$ where $Z[k]$ is the flattened matrix of $\Z$ along mode $k$.  $\Z$ is called a Gaussian tensor with parameters $(\M, \Sig_{1},\ldots, \Sig_{m})$ if 
$Z^{(k)}(:,1), Z^{(k)}(:,2), \ldots, Z^{(k)}(:,m_{k})$ are independent with 
\beq\label{eq:defguasst}
 Z^{(k)}(:, j)\sim \N_{n_{k},m_{k}}(M^{(k)}(:, j), \la_{k}\Sig_{k}), \quad  \forall k\in [m]. 
\eeq
 
\indent We note that the above definition for a general Gaussian tensor reduces to a multivariate Gaussian distribution when $m=1$ and to a Gaussian matrix when $m=2$. 
Obviously a tensor $\Z\in \T_{m,n}$ is a Gaussian tensor if $\Z\bx^{m}$ follows a Gaussian distribution for every nonzero vector $\bx\in \R^{n}$. \par
\indent The following theorem tells that each flattened matrix of an 3-order SND tensor along any direction is a SND matrix. \par 

\begin{thm}\label{th: Gtflattern}
Let $\Z \in  \R^{n_{1}\times \ldots \times n_{m}}$ be a Gaussian tensor with $\textrm{I}:=n_{1}\times \ldots \times n_{m}$. Then $Z[k]\in \R^{n_{k}\times N_{k}}$ is a 
Gaussian matrix for each $k\in [m]$, where $Z[k]$ is the flattened matrix of $\Z$ along the $k$-mode. 
\end{thm} 

\indent Theorem \ref{th: Gtflattern} is directly from Definition \ref{def: sndtensor}. \\  
\indent  Now we consider an $m$-order tensor $\A\in  \T({\rm I})$ of size $\rm{I}:=d_{1}\times d_{2}\times \ldots \times d_{m}$ and denote $a[k,j]$ the $j$th 
fibre of $\A$ along the $k$-mode where $k\in [m]$ and $j$ ranges from 1 to $N_{k}:=d_{1}d_{2}\ldots d_{m}/d_{k}$.  We call $\A$ a \emph{standard Gaussian tensor} if  
$a[k,j]\sim \N_{d_{k}}(0, I_{d_{k}})$ for each $k, j$, and denote $\A\sim \N_{\rm{I}} (0, I_{d_{1}}, \ldots, I_{d_{m}})$.  A random tensor $\A\in  \T({\rm I})$ is said to follow a 
\emph{Gaussian} (or \emph{normal}) distribution if $A_{[k,j]} \sim \N_{I_{k}} (M_{[k,j]}, \Sig_{k})$ for each $k,j$. The following result also applies to a general case. 
 
\begin{thm}\label{th: 3ordergt}
Let $\Y=(Y_{ijk})\in \R^{n_{1}\times n_{2}\times n_{3}}$ be a 3-order random tensor, $\M=(M_{ijk})\in \R^{n_{1}\times n_{2}\times n_{3}}$ be a constant tensor  
and $U_{k}\in \R^{n_{k}\times n_{k}}$ be invertible matrices. If  
\[ \Y = \M + \X \times_{1}U_{1}\times_{2}U_{2}\times_{3}U_{3} \textit{\ with\ }  \X\in \R^{n_{1}\times n_{2}\times n_{3}} \]  
where $\X$ is a standard normal tensor. Then $\A$ follows a Gaussian distribution with parameters $(\M, \Sig_{1}, \Sig_{2}, \Sig_{3})$ where 
$\Sig_{k}=U_{k}U_{k}^{\top}$. 
\end{thm} 

\begin{conj} 
A random tensor $\Y\in  \T_{m;n}$ follows a normal distribution $\Y \sim \N_{m;n}(\M, \Sig_{1},\ldots, \Sig_{m})$ iff  there exist some matrices $U_{k}$ ($k\in [m]$) 
such that  $\Y=\X\times_{1}U_{1}\times_{2}U_{2}\ldots \times_{m}U_{m}$ and $\X$ obeys a standard Gaussian distribution. 
\end{conj}
 
\begin{thm}\label{th: rt2rm} 
A random tensor $\Y\in  \T_{\rm{I}}$ follows a normal distribution $\Y \sim \N_{\rm{I}}(\M, \Sig_{1},\ldots, \Sig_{m})$ iff  
\[ \Y[k] \sim \N_{n_{k}, m_{k}}(\M_{k}, \Sig_{k}, \Om_{k}) \]
where $\Y[k]$ is the unfolding of $\Y$ along mode-$k$ and $\Om_{k}:=\Sig_{m}\otimes \ldots \otimes \Sig_{k+1}\otimes \Sig_{k-1}\ldots \otimes \Sig_{1}$ for each $k\in [m]$. 
\end{thm}

\begin{proof} This is true for $m=1,2$ by the result on random vector and random matrix cases.  Using the unfolding of tensor $\Y$ and induction on $m$, we easily get the result. 
\end{proof}   

\indent  Let $\X\sim \N_{\rm{I}}(0, I_{d_{1}},\ldots, I_{d_{m}} )$ be a random following a standard normal distribution(SND). The density function of $\X$ is defined by 
\[  f_{\X} (\T) = (2\pi)^{-\frac{1}{2} d} \exp(-\frac{1}{2}\seq{\T, \T})  \] 
where $d=d_{1}d_{2}\ldots d_{m}$. 
\begin{thm}
Let $\X\sim \N_{\rm{I}}(0, I_{d_{1}},\ldots, I_{d_{m}} )$ be a random following a SND. Then the CF of $\X$ is 
\[  \phi_{\X}(\T) = \exp(-\frac{1}{2}\seq{\T, \T}) \]
where $\T=(T_{i_{1}\ldots i_{m}})\in \R^{d_{1}\times d_{2}\times \ldots \times d_{m}}$. 
\end{thm}
  
\begin{proof}  By definition of CF, we have 
\beyy 
\phi_{\X}(\T) &=& E[\exp(\imath \seq{\T, \X})] \\
                    &=& E[\exp(\imath \sum\limits_{i_{1},\ldots, i_{m}} T_{i_{1}\ldots i_{m}} X_{i_{1}\ldots i_{m}} )] \\
                    &=& E[\prod\limits_{i_{1},\ldots, i_{m}} \exp(\imath (T_{i_{1}\ldots i_{m}} X_{i_{1}\ldots i_{m}} ))] \\
                    &=& \prod\limits_{i_{1},\ldots, i_{m}} E[\exp(\imath (T_{i_{1}\ldots i_{m}} X_{i_{1}\ldots i_{m}} ))] \\
                    &=& \prod\limits_{i_{1},\ldots, i_{m}} \exp(-\frac{1}{2} T_{i_{1}\ldots i_{m}} ^{2}) \\
                    &=& \exp[-\frac{1}{2} \sum\limits_{i_{1},\ldots, i_{m}} T_{i_{1}\ldots i_{m}} ^{2}] \\ 
                    &=& \exp [-\frac{1}{2} \seq{\T, \T}]
\eeyy
\end{proof}


\section*{Compliance with ethical standards}
\textbf{Conflicts of interest}   On behalf of all authors, the corresponding author states that there is no conflict of interest. 

\vskip 50pt


\end{document}